\documentclass[12pt]{amsart}
\usepackage[utf8]{inputenc}
\usepackage[margin=1in]{geometry}

\usepackage[
    backend=biber,
    style=ieee,
    citestyle=numeric-comp,
    sorting=nyt,
    dashed=false
  ]{biblatex}
\usepackage{hyperref}
\usepackage[hyphenbreaks]{breakurl}
\usepackage{xurl}
\hypersetup{
    colorlinks=true,
    linkcolor=blue,
    citecolor=red,
    breaklinks=true
}

\usepackage{graphicx}
\usepackage{amsmath}
\usepackage{amssymb}
\usepackage{appendix}
\usepackage{amsthm}
\usepackage{booktabs}
\usepackage{graphicx}
\usepackage{subcaption}
\usepackage{comment}
\usepackage{enumerate}
\usepackage{esint}

\usepackage[dvipsnames]{xcolor}

\newtheorem{theorem}{Theorem}[section]
\newtheorem{corollary}{Corollary}[theorem]
\newtheorem{prop}[theorem]{Proposition}
\newtheorem{lemma}[theorem]{Lemma}

\theoremstyle{definition}
\newtheorem{definition}[theorem]{Definition}
\theoremstyle{remark}
\newtheorem{remark}[theorem]{Remark}
\newtheorem{conj}{Conjecture}

\newtheorem{question}{Question}

\usepackage{amsmath}
\usepackage{amssymb}
\usepackage{mathrsfs}

\newcommand{\R}{{\mathbb R}}

\newcommand{\F}{{\mathcal{F}}}

\newcommand{\cS}{{\mathcal{S}}}

\newcommand{\sS}{{\mathscr{S}}}
\newcommand{\fS}{{\mathfrak{S}}}
\newcommand{\tr}{{\text{tr}}}

\usepackage{todonotes}
\usepackage{orcidlink}


\title[Fourier-Wigner multipliers]{Fourier-Wigner multipliers and the Bochner-Riesz conjecture for Schatten class operators}

\author{Helge J. Samuelsen}
\address{Department of Mathematical Sciences,
         Norwegian University of Science and Technology,
         Trondheim, Norway}

\email{helge.j.samuelsen@ntnu.no}
\date{\today}

\begin{document}

\begin{abstract}
In this paper we introduce the notion of Fourier-Wigner multipliers for the Schatten class operators $\cS^p$, which acts as an extension of classical localisation operators in time-frequency analysis. We establish results about Fourier-Wigner multipliers through convolution with rank-one operators and Werner-Young's inequality. We are also able to prove an equivalence relation for compactly supported Fourier multipliers. This allows us to reformulate the Bochner-Riesz conjecture in terms of Schatten class operators.
\end{abstract}
\maketitle

\section{Introduction}
An important tool in classical harmonic analysis is that of Fourier multipliers. Given $m\in L^\infty(\R^{2d})$, we define a linear operator acting on $\Psi\in\mathscr{S}(\R^{2d})$ by
\[
T_m(\Psi)(z)=\int_{\R^{2d}}m(\zeta)\F_\sigma(\Psi)(\zeta)e^{-2\pi i \sigma(z,\zeta)}\,d\zeta.
\]
Here $\sigma$ denotes the standard symplectic form on $\R^{2d}$, and $\F_\sigma$ the symplectic Fourier transform defined by
\[
\F_\sigma(\Psi)(\zeta)=\int_{\R^{2d}}\Psi(z)e^{-2\pi i \sigma(\zeta,z)}\,dz.
\]
The operator $T_m$ can formally be written as 
\[
T_m\Psi=\F_\sigma(m\F_\sigma(\Psi))=\Psi*\F_\sigma(m),
\]
where $\F_\sigma(m)$ is the distributional symplectic Fourier transform of the function $m$. We denote the set of all $m\in L^\infty(\R^{2d})$ for which $T_m$ extends to a bounded linear operator from $L^p(\R^{2d})$ to $L^q(\R^{2d})$ by  $\mathcal{M}_{p,q}(\R^{2d})$. If $p=q$, we simply write $\mathcal{M}_{p,p}(\R^{2d})=\mathcal{M}_p(\R^{2d})$. A well known result states that $\mathcal{M}_{2}(\R^{2d})=L^\infty(\R^{2d})$, and $\mathcal{M}_{1}(\R^{2d})=\F_\sigma(\mathscr{M}(\R^{2d}))$, the symplectic Fourier transform of complex Radon measures. See for instance \cite{Grafakos_14,Grafakos_24,Weiss} for a more in-depth treatment of classical Fourier multipliers.

A famous problem in classical harmonic analysis related to Fourier multipliers is that of the Bochner-Riesz conjecture; for $\delta >0$, the Bochner-Riesz multiplier is given by
\[
m_\delta(z)=\max\left\{(1-|z|^2)^\delta,0\right\},
\]
and the conjecture states that $m_\delta\in \mathcal{M}_{p}(\R^{2d})$ if and only if
\[
\left|\frac{1}{p}-\frac{1}{2}\right|<\frac{2\delta+1}{4d}.
\]
This is connected to Fourier inversion of $L^p$ functions \cite{Stein_91}. We refer to \cite[Chapter $21.3$]{Mattila_15} for a more extensive presentation of the conjecture.

Tao showed that the Bochner-Riesz conjecture implies the Fourier restriction conjecture \cite{Tao_99}. In a recent paper, Luef and the author showed that the restriction conjecture on $\R^{2d}$ can be extended to a quantum version involving the restriction of a suitable Fourier transform of operators in the Schatten $p$-class $\cS^p$ \cite{Luef_Samuelsen_24,Mueller}. This paper aims to extend the Bochner-Riesz conjecture in a similar fashion, and acts as a continuation of recent work on extending classical harmonic results to the setting of quantum harmonic analysis, see \cite{Luef_Skrettingland_21,Fulsche_Luef_Werner_24,Luef_Samuelsen_24,Samuelsen_24}. 

We define the symmetric time-frequency shift $\rho:\R^{2d}\to\mathcal{L}(L^2(\R^d))$ acting on $f\in L^2(\R^d)$ at a point $z=(x,\xi)\in\R^{2d}$ by
\begin{equation}\label{eq:SchRep}
\rho(x,\xi)f(t)=e^{-\pi i x\cdot\xi}e^{2\pi i\xi\cdot t}f(t-x).
\end{equation}
This is a projective unitary representation of phase space acting on $L^2(\R^d)$. Given a trace class operator $T\in \cS^1$, the Fourier-Wigner transform of $T$ at a point $z\in\R^{2d}$ is defined by
\begin{equation}\label{eq:DefFW}
\F_W(T)(z):=\tr(T\rho(z)^*)=\tr(T\rho(-z)).
\end{equation}
The Fourier-Wigner transform of a trace class operator is a continuous function on phase space vanishing at infinity by an analogue of the Riemann-Lebesgue lemma \cite{Werner}. Moreover, it extends to a $*$-isometric isomorphism between the Hilbert-Schmidt operators and $L^2(\R^{2d})$ due to a result by Pool \cite{Pool}. The inverse is given by
\[
\F_W^{-1}(F)=\int_{\R^{2d}}F(z)\rho(z)\,dz, \quad F\in L^1(\R^{2d}),
\]
where the integral is considered weakly, and $F$ is called the spreading function of the operator. The inverse Fourier-Wigner transform is well-studied in \cite{Folland_Phase} where it goes by the integrated Schr\"{o}dinger representation of $F$.

For $m\in L^\infty(\R^{2d})$ we define the analogue of the symplectic Fourier multipliers, which we call Fourier-Wigner multipliers, by
\[
\mathfrak{T}_m(T):= \F_W^{-1}(m\F_W(T))=\int_{\R^{2d}}m(z)\F_W(T)(z)\rho(z)\,dz,
\]
on the class $\cS^1$. Here the integral is considered weakly. Fourier multipliers have been studied for noncommutative spaces \cite{Ruzhansky2024multipliers}, where $L^p$ to $L^q$ estimates were established whenever the multiplier is in a certain Lorentz space. Fourier multipliers is also used in \cite{Lai25} to study the Riesz transform on quantum Euclidean spaces.

When $T$ is a rank-one operator, and $m=\chi_\Omega$ for $\Omega\subseteq \R^{2d}$, the Fourier-Wigner multiplier corresponds to the classical localisation operator first introduced by Daubechies \cite{Daubechies}. Localisation operators have been studied in detail and seen applications to signal processing \cite{Cordero, Kreme, Olivero, Rajhamshi, Strohmer}.

We denote by $\mathfrak{M}_{p,q}(\R^{2d})$ the set of all $m\in L^\infty(\R^{2d})$ where $\mathfrak{T}_m$ extends to a bounded linear operator from $\cS^p$ to $\cS^q$. As in the classical case, we write $\mathfrak{M}_{p,p}(\R^{2d})=\mathfrak{M}_{p}(\R^{2d})$. Whenever $m$ is compactly supported, we have the following identification.
\begin{theorem}\label{thm:MainThm}
Let $m\in L^\infty(\R^{2d})$ be compactly supported. Then for each $1\leq p,q\leq \infty$ it follows that $m\in \mathcal{M}_{p,q}(\R^{2d})$ if and only if $m\in \mathfrak{M}_{p,q}(\R^{2d})$. 
\end{theorem}
Theorem \ref{thm:MainThm} allows us to extend well-known classical Fourier multiplier results to the quantum setting for compactly supported multipliers. For instance, combining the ball multiplier theorem of Fefferman \cite{Fefferman_71} with Theorem \ref{thm:MainThm} gives the following extension of Fefferman's ball multiplier theorem.
\begin{corollary}
Let $m=\chi_{B_1(0)}$ be the characteristic function of the unit ball in $\R^{2d}$. Then $m\in \mathfrak{M}_p(\R^{2d})$ if and only if $p=2$.
\end{corollary}
As the Bochner-Riesz multiplier is also supported on the unit ball, we can reformulate the Bochner-Riesz conjecture in terms of Fourier-Wigner multipliers.
\begin{conj}[Quantum Bochner-Riesz conjecture]
For $\delta>0$, the multiplier $m_\delta$ is in $\mathfrak{M}_p(\R^{2d})$ if and only if
\[
\left|\frac{1}{p}-\frac{1}{2}\right|<\frac{2\delta+1}{4d}
\]
\end{conj}

Using that the quantum Bochner-Riesz conjecture is equivalent to the classical Bochner-Riesz conjecture on phase space by Theorem \ref{thm:MainThm}, it follows by Corollary $1.3$ in \cite{Tao_99} that the quantum Bochner-Riesz conjecture in fact implies the classical restriction conjecture on $\R^{2d}$. The quantum and classical restriction problems on phase space have been shown to be equivalent for compactly supported measures \cite{Luef_Samuelsen_24}. As such, we have the following implication.
\begin{corollary}
The quantum Bochner-Riesz conjecture implies the quantum Fourier restriction conjecture for the unit sphere.
\end{corollary}

\section{Preliminaries}
\subsection{Time-frequency analysis}
Let $\mu$ be a finite Borel measure on $\R^{2d}$, and define the symplectic Fourier transform at a point $\zeta\in\R^{2d}$ by
\[
\F_\sigma(\mu)(\zeta)=\int_{\R^{2d}}e^{-2\pi i \sigma(\zeta,z)}d\mu(z)\in C_b(\R^{2d}).
\]
Here $\sigma$ denotes the standard symplectic form on $\R^{2d}$ given by
\[
\sigma\left((x,\xi),(x',\xi')\right)=x'\cdot\xi-x\cdot\xi',\qquad (x,\xi),(x',\xi')\in\R^{2d}.
\]
The symplectic Fourier transform extends to a unitary operator on $L^{2}(\R^{2d})$ with the property that $\F_\sigma^2=\mathrm{Id}_{L^2}$.

Define the symmetric time-frequency shift $\rho:\mathbb{R}^{2d}\to \mathcal{L}\left(L^2(\R^d)\right)$  by \eqref{eq:SchRep}. Then it gives rise to a time-frequency distribution known as the \textit{cross-ambiguity function}. Namely, given $f,g\in L^2(\R^d)$, the cross-ambiguity function is defined as
\begin{equation}\label{AmbiguityDef}
\mathcal{A}(f,g)(z):=\langle f,\rho(z)g\rangle_{L^2},\qquad z\in\R^{2d}.
\end{equation}
Moyal's identity ensures that $\mathcal{A}(f,g)\in L^2(\R^{2d})$ whenever $f,g\in L^2(\R^d)$ \cite[Chapter $3$]{Grochenig}. 
If $f$ and $g$ are both the $L^2$-normalised Gaussian $\varphi_0(t)=2^{d/4}\exp(-\pi|t|^2)$, then the cross-ambiguity function is the Gaussian,
\begin{equation}\label{Ambi-Gauss}
\mathcal{A}\left(\varphi_0,\varphi_0\right)(z)=e^{-\pi\frac{|z|^2}{2}},
\end{equation}
with $z\in\R^{2d}$. The cross-ambiguity function satisfies a covariance property.
\begin{lemma}[Lemma $2.5$ in \cite{Luef_Samuelsen_24}]\label{Covariance property}
    Let $f,g\in L^2(\R^d)$. Then for any $z,\zeta\in \R^{2d}$
    \begin{equation*}
        \mathcal{A}(\rho(\zeta)f,g)(z)=e^{\pi i \sigma(\zeta,z)}\mathcal{A}(f,g)(z-\zeta).
    \end{equation*}
\end{lemma}

\subsection{Distributions}
We denote the Schwartz class by $\sS(\R^{2d})$, and its continuous dual space of tempered distributions by $\sS'(\R^{2d})$. For $\tau\in\sS'(\R^{2d})$ and $\psi\in \sS(\R^{2d})$, we use the sesquilinear dual pairing
\[
\langle \tau,\psi\rangle_{\sS',\sS}:=\tau(\overline{\psi}),
\]
in order to stay consistent with the $L^2$ inner product. Namely, for any $f\in L^2(\R^{2d})$,
\[
\langle f,\psi\rangle_{\sS',\sS}=\int_{\R^{2d}}f(x)\overline{\psi(x)}\,dx=\langle f,\psi\rangle_{L^2},
\]
for all $\psi\in \sS(\R^{2d})$. The symplectic Fourier transform can be extended to an isomorphism on the tempered distributions through
\[
\langle \F_\sigma(\tau),\psi\rangle_{\sS',\sS}:=\langle \tau,\F_\sigma(\psi)\rangle_{\sS',\sS},
\]
for $\tau\in \sS'(\R^{2d})$ and all $\psi\in \sS(\R^{2d})$. Moreover, by restricting the windows to the Schwartz class, the cross-ambiguity function can be extended to be valid for tempered distributions through
\[
\mathcal{A}(\tau,\psi)(z):=\langle \tau,\rho(z)\psi\rangle_{\sS',\sS},
\]
for $\tau\in\sS'(\R^d)$ and $\psi\in \sS(\R^d)$.

Modulation spaces are of great importance in the theory of time-frequency analysis. For $1\leq p,q\leq\infty$, we define the modulation space $M^{p,q}(\R^d)$ through the ambiguity function by considering all tempered distributions $\tau\in\sS'(\R^d)$ such that,
\[
\|\tau\|_{M^{p,q}}:=\|\mathcal{A}(\tau,\varphi_0)\|_{L^q_\xi L^p_x}<\infty,
\]
where $\varphi_0(t)=2^{d/4}\exp{(-\pi|t|^2)}$ is the $L^2$-normalised Gaussian. A standard reference for modulation spaces is \cite[Chapter $12$]{Grochenig}.

\subsection{Operator theory}
We denote the space of bounded linear operators on $L^2(\R^d)$ by $\mathcal{L}(L^2(\R^d))$ and compact operators by $\mathcal{K}$. Any compact operator $T$ can be represented as 
\[
T=\sum_{n\in\mathbb{N}}s_n(T)e_n\otimes \eta_n,
\]
where $e_n\otimes \eta_n (f)=\langle f,\eta_n\rangle e_n$ is a rank-one operator. This is known as the singular value decomposition of $T$, where $\{s_n(T)\}\subseteq [0,\infty)$ are the singular values of $T$ and $\{e_n\}_{n\in\mathbb{N}},\{\eta_n\}_{n\in\mathbb{N}}\subseteq L^2(\R^d)$ are two orthonormal families. Given $1\leq p\leq \infty$, the Schatten $p$-class of compact operators can be defined through the singular value decomposition by
\[
\cS^p:=\{T\in\mathcal{L}(L^2(\R^d)):\{s_n(T)\}_{n\in\mathbb{N}}\in \ell^p\}.
\]
Whenever $p=\infty$, we make the identification $\cS^\infty=\mathcal{L}(L^2(\R^d))$. Since the Schatten $p$-classes are defined through $\ell^p$, the following inclusion property holds; $\cS^1\subseteq \cS^p\subseteq \cS^q\subseteq \mathcal{K}$ for $q>p$. The Schatten $p$-class can be equipped with the norm
\[
\|T\|_{\cS^p}:=\left(\sum_{n\in\mathbb{N}}|s_n(T)|^p\right)^\frac{1}{p}, \qquad 1\leq p<\infty.
\]
This norm turns $\cS^p$ into a Banach space, and the associated dual space is identified with $\cS^{p'}$, where $p'$ is the H\"{o}lder conjugate of $p$ whenever $1<p< \infty$. A version of the Riesz-Thorin interpolation theorem for Schatten classes can be found in \cite[Theorem $7.13.2$]{Sukochev_NCInt}, and we refer to \cite[Chapter $3$]{SimonOp} or \cite{SimonTrace} for more details on Schatten classes. 
\begin{theorem}[Riesz-Thorin Interpolation for Operators]\label{thm:RTInter}
Let $1\leq p_i,q_i\leq \infty$ for $i=0,1$. Suppose $\mathfrak{T}$ is a linear operator from $\cS^{p_0}+\cS^{p_1}$ into $\cS^{q_0}+\cS^{q_1}$ with the estimate
\[
\|\mathfrak{T}(T)\|_{\cS^{q_i}}\leq C_i\|T\|_{\cS^{p_i}}\qquad \text{for } i=0,1.
\]
Then for $0<\theta<1$ and
\[
\frac{1}{p}=\frac{1-\theta}{p_0}+\frac{\theta}{p_1},\qquad \frac{1}{q}=\frac{1-\theta}{q_0}+\frac{1}{q_1},
\]
we have
\[
\|\mathfrak{T}(T)\|_{\cS^{q}}\leq C_0^{1-\theta} C_1^{\theta}\|T\|_{\cS^{p}}.
\]
\end{theorem}

A further generalisation of Schatten classes are the Lorentz Schatten classes of compact operators on $L^2(\R^d)$. The Lorentz quasi-norm on a general measure space $(X,\mathcal{F},\mu)$ is defined by
\[
\|f\|_{L^{p,q}}:=\left(\frac{q}{p}\int_0^\infty t^{\frac{q}{p}-1}|\inf_{\alpha>0}\{\mu(|f(x)|>\alpha|)\leq t\}|^q\,dt\right)^\frac{1}{q},
\]
for $1\leq p,q<\infty$. For $q=\infty$, the quasi-norm is defined by
\[
\|f\|_{p,\infty}=\|f\|_{p,w}=\sup_{t\geq 0}\left(t^{\frac{1}{p}}\inf_{\alpha>0}\{\mu(|f(x)|>\alpha|)\leq t\}\right),
\]
and coincides with the weak $L^p$ spaces. We identify $L^{\infty,\infty}(X)$ with $L^\infty(X)$. The Lorentz spaces are nested as 
\[
\|f\|_{L^{p,q_2}(X)}\leq \|f\|_{L^{p,q_1}(X)}
\]
for $q_1\leq q_2$, see for instance \cite{Weiss}. We also mention H\"{o}lder's inequality for Lorentz spaces.
\begin{theorem}[Theorem $3.4$ in \cite{ONeil}]
Let $f\in L^{p_1,q_1}(X)$ and $g\in L^{p_2,q_2}(X)$ for
\[
\frac{1}{p_1}+\frac{1}{p_2}<1,\qquad 0<q_1,q_2\leq\infty.
\]
Then $fg\in L^{p,q}(X)$ where
\[
\frac{1}{p}=\frac{1}{p_1}+\frac{1}{p_2},\qquad \frac{1}{q}\leq\frac{1}{q_1}+\frac{1}{q_2}.
\]
Moreover, there exists $C=C(p_1,p_2,q_1,q_2)>0$ such that
\[
\|fg\|_{L^{p,q}}\leq C\|f\|_{L^{p_1,q_1}}\|g\|_{L^{p_2,q_2}}.
\]
\end{theorem}

Restricting to Lorentz sequence spaces, and using that singular values are arranged in a non-increasing order, allows us to define the Lorentz Schatten classes $\cS^{p,q}$ by the quasi-norms,
\[
\|T\|_{\cS^{p,q}}:=\left(\sum_{n=1}^\infty n^{\frac{q}{p}-1}|s_n(T)|^q\right)^\frac{1}{q},
\]
and the weak $S^{p,\infty}$-quasi-norm
\[
\|T\|_{\cS^{p,\infty}}:=\sup_{n\in\mathbb{N}}\{n^\frac{1}{p}s_{n}(T)\}.
\]

There is a useful version of the Hunt-Marcinkiewicz interpolation theorem for Schatten class operators. This was first proved for weak Schatten classes by Simon \cite{SimonInterpolation}, and can also be found in \cite[Theorem $2.10$]{SimonTrace}. For two fixed orthonormal families $\{\varphi_n\},\{\psi_n\}\subseteq L^2(\R^d)$, the idea of Simon is first to restrict to the closed subspace $\cS^{p,q}(\{\varphi_n\},\{\psi_n\})\subseteq \cS^{p,q}$ of operators of the form
\[
T=\sum_{n=1}^\infty s_n(T)\varphi_{n}\otimes \psi_n.
\]
From here we can define a linear isomorphism between $\cS^{p,q}(\{\varphi_n\},\{\psi_n\})$ and $\ell^{p,q}$, and the problem is reduced to the classical Hunt-Marcinkiewicz interpolation theorem for Lorentz spaces, which can be found in \cite{Hunt,Weiss}. A version for noncommutative Euclidean spaces can be found in \cite[Theorem $7.5.4$]{Sukochev_NCInt}. 
\begin{theorem}[Hunt-Marcinkiewicz Interpolation for Operators]\label{Thm:Hunt_Marc_Interpolation}
Suppose $\mathfrak{T}$ is a linear operator from $\cS^{p_0,r_0}+\cS^{p_1,r_1}$ into $L^{q_0,s_0}(X)+L^{q_1,s_1}(X)$ for $r_0<r_1$ and $p_0\neq p_1$, and there is the estimate
\[
\|\mathfrak{T}(T)\|_{L^{q_i,s_i}(X)}\leq \|T\|_{\cS^{p_i,r_i}}\qquad \text{for } i=0,1.
\]
Then there exists a constant $C=C(\theta)>0$ such that
\[
\|\mathfrak{T}(T)\|_{L^{p,q}(X)}\leq C\|T\|_{\cS^{r,q}},
\]
where $1\leq q\leq \infty$ and
\[
\frac{1}{p}=\frac{1-\theta}{p_0}+\frac{\theta}{p_1},\qquad \frac{1}{r}=\frac{1-\theta}{r_0}+\frac{1}{r_1}, \qquad 0<\theta<1.
\]
\end{theorem}
\begin{remark}
The construction of the closed subspaces $S^{p,q}(\{\varphi_n\},\{\psi_n\})$ allows for a version of the Hunt-Marcinkiewicz interpolation theorem to hold for linear maps between Lorentz Schatten classes by choosing the Lorentz sequence spaces $L^{q,r}(X)=\ell^{q,r}$ in Theorem \ref{Thm:Hunt_Marc_Interpolation}.
\end{remark}

\subsection{Fourier analysis of operators}
Recall that $\rho:\R^{2d}\to \mathcal{L}(L^2(\R^d))$ is the symmetric time-frequency shift defined in \eqref{eq:SchRep}. Given $T\in\cS^1$, we define the Fourier-Wigner transform at a point $z\in\R^{2d}$ by
\[
\F_W(T)(z)=\tr(T\rho(-z))=\langle T,\rho(z)\rangle_{\cS^1,\cS^\infty}.
\]
For $F\in L^1(\R^{2d})$, the inverse Fourier-Wigner transform, also known as the integrated Schr\"{o}dinger operator, is given by
\[
T=\F_W^{-1}(F)=\int_{\R^{2d}}F(z)\rho(z)\,dz.
\]
This is a continuous linear map from $L^1(\R^{2d})$ into $\mathcal{K}$ by  \cite[Theorem $1.30$]{Folland_Phase}. 

For rank-one operators $T=g\otimes h$ with $g,h\in L^2(\R^d)$, the Fourier-Wigner transform is given by 
\begin{equation}\label{eq:FW_Ambi}
\F_W(g\otimes h)(z)=\mathcal{A}(g,h)(z),
\end{equation}
where $\mathcal{A}(g,h)$ is the cross-ambiguity function defined in \eqref{AmbiguityDef}.

If $T\in \cS^1$, then $\F_W(T)$ belongs to the space of continuous functions vanishing at infinity, denoted $C_0(\R^{2d})$. This is a quantum analogue of the Riemann-Lebesgue lemma and was first shown in \cite{Werner}. Due to a result by Pool \cite{Pool}, the Fourier-Wigner transform extends to a unitary operator from $\cS^2$ to $L^2(\R^{2d})$. Therefore, by Theorem \ref{Thm:Hunt_Marc_Interpolation} we regain Hausdorff-Young's inequality for operators.
\begin{prop}[Quantum Hausdorff-Young]\label{prop:QHY}
Let $1\leq p\leq 2$, and $p'=p/(p-1)$ be the H\"{o}lder conjugate. Then for any $T\in \cS^{p}$,
\[
\|\F_W(T)\|_{L^{p',p}(\R^{2d})}\leq C \|T\|_{\cS^{p}}.
\]
If $T\in \cS^{p,p'}$, then
\[
\|\F_W(T)\|_{L^{p'}(\R^{2d})}\leq C\|T\|_{\cS^{p,p'}}.
\]
\end{prop}
\begin{prop}[Reverse Quantum Hausdorff-Young]
Let $1\leq p\leq 2$, and $p'=p/(p-1)$ be the H\"{o}lder conjugate. Then for any $F\in L^{p}(\R^{2d})$,
\[
\|\F_W^{-1}(F)\|_{\cS^{p',p}}\leq C\|F\|_{L^{p}(\R^{2d})}.
\]
If $F\in L^{p,p'}(\R^{2d})$, then
\[
\|\F_W^{-1}(F)\|_{\cS^{p'}}\leq C\|F\|_{L^{p,p'}(\R^{2d})}.
\]
\end{prop}
\begin{remark}
The version of Hausdorff-Young presented in Proposition \ref{prop:QHY} is sharper than the version presented in \cite{Werner,Luef_Skrettingland_19} by the inclusion of Lorentz spaces.  Namely, for $1\leq p\leq 2$, it follows that
\[
\|\F_W(T)\|_{L^{p'}}=\|\F_W(T)\|_{L^{p',p'}}\leq C\|\F_W(T)\|_{L^{p',p}}\leq C\|T\|_{\cS^p},
\]
which is, up to a constant, the Hausdorff-Young inequality often presented in quantum harmonic analysis.
\end{remark}

\subsection{Operator convolutions and localisation operators}
In what follows, $Pf(x)=f(-x)$ denotes the parity operator, and $\alpha_z(T)=\rho(z)T\rho(-z)$ denotes conjugation with the symmetric time-frequency shift. For $S,T\in \cS^1$ we define an operator-operator convolution by
\[
T\star S(z):=\tr(T\alpha_z(PSP)),
\]
which is a function on phase space. Given $T\in\cS^1$ and $F\in L^1(\R^{2d})$,
we also introduce function-operator convolution by
\[
F\star T=T\star F:=\int_{\R^{2d}}F(z)\alpha_z(T)\,dz,
\]
where the integral is considered weakly. These convolutions were introduced in \cite{Werner}, where a version of Young's inequality for convolutions was also proved. We state the inequality as given in \cite{Luef_Skrettingland_19}.
\begin{theorem}[Werner-Young's inequality]\label{Werner-Young}
Let $1\leq p,q,r\leq \infty$ be such that $1+r^{-1}=p^{-1}+q^{-1}$. If $F\in L^p(\R^{2d})$, $T\in \cS^p$ and $S\in \cS^q$, then $F\star S\in \cS^r$ and $T\star S\in L^{r}(\R^{2d})$. Moreover, there are the norm bounds
\begin{align*}
\|F\star S\|_{\cS^r}\leq& \|F\|_{L^p}\|S\|_{\cS^q},\\
\|T\star S\|_{L^r}\leq& \|T\|_{\cS^p}\|S\|_{\cS^q}.
\end{align*}
\end{theorem}

It is worth noting that these convolutions interact with the Fourier-Wigner transform in the same way as the classical convolution interacts with the symplectic Fourier transform.
\begin{theorem}[Proposition $3.12$ in \cite{Luef_Skrettingland_19}]
Let $F\in L^1$, $S,T\in \cS^1$. Then the following holds,
\begin{align*}
\F_W(F\star S)=\F_\sigma(F)\F_W(S),\\
\F_\sigma(T\star S)=\F_W(T)\F_W(S).
\end{align*}
\end{theorem}

When the operator $T=\varphi\otimes \psi$ is a rank-one operator with $\varphi,\psi\in L^2(\R^d)$, then the convolution with $F\in L^p(\R^{2d})$, for any $1\leq p\leq \infty$, gives 
\begin{align*}
\langle F\star (\varphi\otimes \psi) g,h\rangle=&\,\int_{\R^{2d}} F(z)\langle (\varphi\otimes \psi)\rho(-z)g,\rho(-z)h\rangle \,dz\\
=&\,\int_{\R^{2d}} F(z)\langle g,\rho(z)\psi\rangle\langle \rho(z)\varphi,h\rangle \,dz\\
=&\,\int_{\R^{2d}} F(z)\mathcal{A}(g,\psi)(z)\overline{\mathcal{A}(h,\varphi)}(z) \,dz\\
=&\langle \mathcal{A}_{F}^{\psi,\varphi}g,h\rangle,
\end{align*}
for any choice of $g,h\in L^2(\R^d)$. The operator $F\star (\varphi\otimes \psi)=\mathcal{A}_F^{\psi,\varphi}$ is known as a mixed-state localisation operator with symbol $F$ and windows $\psi$ and $\varphi$. By Werner-Young's inequality, it follows that the localisation operator is in $\cS^p$ whenever $F\in L^p(\R^{2d})$. A necessary condition was established in \cite{Cordero2}.
\begin{theorem}[Theorem $1$ in \cite{Cordero2}]\label{thm:NessSchattenCond}
Let $F\in\sS'(\R^{2d})$ and $1\leq p\leq \infty$. Assume that there is a constant $C>0$ such that
\[
\|F\star(\varphi\otimes\psi)\|_{\cS^p}\leq C\|\varphi\|_{M^{1,1}}\|\psi\|_{M^{1,1}}.
\]
for every $\psi,\varphi\in \sS(\R^d)$. Then $F\in M^{p,\infty}(\R^{2d})$.
\end{theorem}

\subsection{Weyl quantisation and Schwartz operators}
Any continuous linear operator $T:\mathscr{S}(\R^d)\to \sS'(\R^d)$ can be associated to a tempered distribution $a_T\in\mathscr{S}'(\R^{2d})$, called the Weyl symbol of $T$, which is defined through the relation
\begin{equation}\label{WeylQuant}
\langle T \psi,\varphi\rangle_{\sS',\sS}=\langle a_T, \mathcal{W}(\varphi,\psi)\rangle_{\sS',\sS},
\end{equation}
for all $\psi,\varphi\in\mathscr{S}(\R^d)$. Here $\mathcal{W}$ denotes the cross-Wigner distribution defined by
\[
\mathcal{W}(\varphi,\psi)(x,\xi):=\int_{\R^d}\varphi\left(x+\frac{t}{2}\right)\overline{\psi\left(x-\frac{t}{2}\right)}e^{-2\pi i t\cdot\xi}\,dt=\F_\sigma\left(\mathcal{A}(\varphi,\psi)\right)(x,\xi).
\]
For any $a\in\sS'(\R^{2d})$, we can define a continuous linear map $L_a:\sS(\R^{d})\to\sS'(\R^d)$ through \eqref{WeylQuant}. Thus, there is an isomorphism between $\sS'(\R^{2d})$ and $\mathcal{L}(\sS(\R^d);\sS'(\R^d))$. Using this isomorphism, it is possible to define the Schwartz operators as those operators with Weyl symbol in the Schwartz class. These operators were studied in \cite{Keyl-Kiukas-Werner_16}.
\begin{definition}
Let $\mathfrak{S}$ denote the space of all continuous operators $L_a:\mathscr{S}(\R^{d})\to \mathscr{S}'(\R^{d})$ for which the Weyl symbol is in $\mathscr{S}(\R^{2d})$, e.g.
\begin{equation*}
\mathfrak{S}:=\left\{L_a:\mathscr{S}(\R^d)\to\mathscr{S}'(\R^d): a\in \mathscr{S}(\R^{2d})\right\}.
\end{equation*}
\end{definition}
The Schwartz class of operators plays the role of a suitable good subspace of operators in quantum harmonic analysis similar to that of Schwartz functions in classical harmonic analysis. This claim is justified by the following observation.
\begin{lemma}[\cite{Keyl-Kiukas-Werner_16}, Lemma $2.5$]\label{lem:densityOfSchwartz}
The Schwartz class $\fS$ is dense in $\cS^p$ for $1\leq p<\infty$.
\end{lemma}
\begin{proof}
The class $\fS$ is a subspace of $\cS^1$ by Proposition $4.1$ in \cite{Grochenig-Heil}, and consequently a subspace of $\cS^p$ for each $p\leq \infty$. As the finite rank operators are dense in $\cS^p$ for $p<\infty$, it suffices to show that $\mathfrak{S}$ is dense in $\cS^1$. However, this follows by the density of $\sS(\R^d)$ in $L^2(\R^d)$. To see this, we start by showing that rank-one operators can be approximated by operators in $\fS$. By the density of $\sS(\R^d)$ in $L^2(\R^d)$, it follows that for every $f,g\in L^2(\R^d)$ and for each $\delta>0$, there exist $\varphi,\psi\in \sS(\R^d)$ such that
\[
\|f-\varphi\|_{L^2}<\delta,\qquad \|g-\psi\|_{L^2}<\delta.
\]
Whence, it follows from the triangle inequality that
\begin{align*}
\|f\otimes g-\varphi\otimes\psi\|_{\cS^1}\leq&\, \|(f-\varphi)\otimes g\|_{\cS^1}+\|\varphi\otimes (g-\psi)\|_{\cS^1}\\
<&\,(\|g\|_{L^2}+\|\varphi\|_{L^2})\delta<(2+\delta)\delta.
\end{align*}
The Weyl symbol of $\varphi\otimes \psi$ is $\mathcal{W}(\varphi,\psi)\in \sS(\R^{2d})$ and thus $\varphi\otimes \psi\in\fS$. This shows that each rank-one operator can be approximated by operators in $\fS$.

To extend to finite rank operators, let $\varepsilon>0$ and consider any non-zero finite rank operator $T$. As $T$ is of finite rank, there exists $N\in\mathbb{N}$ and a finite collection of orthonormal functions $\{f_n\}_{n=1}^N,\{g_n\}_{n=1}^N\subseteq L^2(\R^d)$ and $\{s_n(T)\}_{n=1}^N\subseteq \mathbb{C}$ such that
\[
T=\sum_{n=1}^{N}s_n(T)f_n\otimes g_n.
\]
This is the singular value decomposition of $T$.  Moreover, for each $\delta>0$ there exists $\{\varphi_n\}_{n=1}^N,\{\psi_n\}_{n=1}^N\subset \sS(\R^d)$ such that
\[
\|f_n-\varphi_n\|_{L^2}<\delta,\qquad \|g_n-\psi_n\|_{L^2}<\delta,
\]
for each $1\leq n\leq N$. Choose $\delta>0$ such that
\[
\delta(2+\delta)<\frac{\varepsilon}{\|T\|_{\cS^1}},
\]
and consider the operator
\[
T'=\sum_{n=1}^N s_n(T)\varphi_n\otimes \psi_n\in\fS,
\]
where $\varphi_n$ and $\psi_n$ are associated to the choice of $\delta$. The operator $T'$ belongs to $\fS$ since the Weyl symbol is a linear combination of Schwartz functions. Using the triangle inequality, we can therefore conclude that
\[
\|T-T'\|_{\cS^1}\leq \sum_{n=1}^{N}|s_n(T)|\|f_n\otimes g_n-\varphi_n\otimes \psi_n\|_{\cS^1}\leq \delta(2+\delta)\sum_{n=1}^N|s_n(T)|<\varepsilon,
\]
which shows that any finite rank operator can be approximated by operators in $\fS$.
\end{proof}
\begin{remark}
The proof of Lemma \ref{lem:densityOfSchwartz} shows that $\fS$ is in fact a dense subspace of the compact operators $\mathcal{K}$ with respect to the norm topology.
\end{remark}

We denote by $\fS'$ the tempered operators, which is the continuous dual space of $\fS$. By Schwartz' kernel theorem, we can identify $\fS'$ by the space of all continuous linear maps from $\sS(\R^d)$ to $\sS'(\R^d)$, see for instance \cite[Theorem $14.3.5$]{Grochenig}. This allows us to give the following definition of the Fourier-Wigner transform for tempered operators.
\begin{definition}\label{defDistFW}
For $T\in\mathfrak{S}'$ we define $\F_W(T)$ to be the tempered distribution such that
\[
\langle T,S\rangle_{\mathfrak{S}',\mathfrak{S}}=\langle\F_W(T),\F_W(S)\rangle_{\mathscr{S}',\mathscr{S}},
\]
holds for all $S\in \mathfrak{S}$.
\end{definition}
As in the case of tempered distributions, it is also true that Schwartz operators are $w^*$ dense in the tempered operators.
\begin{lemma}\label{lem:w*DenseSchwartz}
The space of Schwartz operators $\fS$ is dense in the tempered operators $\fS'$ with respect to the $\sigma(\fS',\fS)$ topology.
\end{lemma}
\begin{proof}
Recall that $\sS(\R^{2d})$ is $w^*$-dense in $\sS'(\R^{2d})$. Thus, for any $T\in \fS'$ there exists $\{\Phi_n\}_{n\in\mathbb{N}}\subset \sS(\R^{2d})$ such that $\F_\sigma(\Phi_n)\stackrel{\ast}{\rightharpoonup}\F_W(T)$. In particular, for any $S\in\fS$, it follows that
\[
\lim_{n\to\infty}\langle L_{\Phi_n},S\rangle_{\fS',\fS}=\lim_{n\to\infty}\langle \F_{\sigma}(\Phi_n),\F_W(S)\rangle_{\sS',\sS}=\langle \F_W(T),\F_W(S)\rangle_{\sS',\sS}=\langle T,S\rangle_{\fS',\fS},
\]
but $L_{\Phi_n}\in \fS$ as $\Phi_n\in \sS(\R^{2d})$ for each $n\in\mathbb{N}$.
\end{proof}
The following result was originally proved in section $5$ of \cite{Keyl-Kiukas-Werner_16} as several different results, and summarised in \cite{Luef_Skrettingland_19}.
\begin{prop}[\cite{Luef_Skrettingland_19},Proposition $3.16$]\label{ConvProp-Distribution}
Let $S,T\in\fS$, $\Psi\in\sS(\R^{2d})$, $\tau\in\sS'(\R^{2d})$ and $A\in \fS'$. Then:
\begin{enumerate}[i)]
\item The following convolution relations hold,
\begin{align*}S\star T\in \sS(\R^{2d}),&\qquad \Psi\star S\in\fS,\\
S\star A\in \sS'(\R^{2d}),&\qquad \Psi\star A\in \fS',\\
\Psi*\tau\in\sS'(\R^{2d}),&\qquad S\star \tau\in \fS'.
\end{align*}
\item $\F_W$ extends to a topological isomorphism from $\fS'$ to $\sS'(\R^{2d})$.
\item The relations \begin{align*}
\F_\sigma(S\star A)=&\,\F_W(S)\F_W(A),\, &&\F_W(\psi\star A)=\F_\sigma(\psi)\F_W(A),\\
\F_\sigma(\Psi*\tau)=&\,\F_\sigma(\Psi)\F_\sigma(\tau),\, &&\F_W(S\star \tau)=\F_W(S)\F_\sigma(\tau),
\end{align*}
hold.
\item The Weyl symbol of $A$ is given by $a_A=\F_\sigma(\F_W(A))$.
\end{enumerate}
\end{prop}

We end the section with the following lemma from \cite{Samuelsen_24}, which is proved by a quantum version of Wiener's division lemma.
\begin{lemma}[\cite{Samuelsen_24}, Corollary $3.2$]\label{LW-bnd-cor}
Let $\Omega\subseteq \R^{2d}$ be a bounded set, and $1\leq p\leq \infty$. There exists $L^2$-normalised $g,h\in \sS(\R^{d})$ and $C=C(\Omega)>0$ such that 
\begin{enumerate}[i)]
\item If $\tau\in \sS'(\R^{2d})$ and $\F_\sigma(\tau)$ is supported on $\Omega$, then
\[
\|\tau\|_{L^p(\R^{2d})}\leq C(\Omega)\|\tau\star(g\otimes h)\|_{\cS^p}.
\]
\item If $T\in\fS'$ and $\F_W(T)$ is supported on $\Omega$, then
\[
\|T\|_{\cS^p}\leq C(\Omega) \|T\star(g\otimes h)\|_{L^p(\R^{2d})}.
\]
\end{enumerate}
\end{lemma}

\section{Fourier-Wigner multipliers and their basic properties}
Given a function $m\in L^\infty(\R^{2d})$, we can define an operator acting on trace class operators through the following construction. Let $T\in\cS^1$, and consider the weakly defined integral
\begin{equation}\label{multiDef}
\mathfrak{T}_m T=\int_{\R^{2d}}m(z)\F_W(T)(z)\rho(z)\,dz.
\end{equation}
We denote by $\mathfrak{T}_m$ the operator associated to the Fourier-Wigner multiplier $m$. By Proposition \ref{ConvProp-Distribution} it follows that the operator $\mathfrak{T}_m$ can we written as the convolution operator 
\[
\mathfrak{T}_m(T)=T\star \F_\sigma(m),
\]
for each $T\in\fS$. Here $\F_\sigma(m)$ denotes the distributional symplectic Fourier transform of $m$. A key observation is that the Fourier-Wigner multipliers commute with convolutions. 
\begin{lemma}\label{lem:Commutation}
For any $T\in \fS'$, $S\in\fS$ and $m\in L^\infty(\R^{2d})$, we have
\[
\mathfrak{T}_m(T)\star S=T_m(T\star S)
\]
in $\sS'(\R^{2d})$.
\end{lemma}
\begin{proof}
This follows from a straightforward computation. For any $\Psi\in \sS(\R^{2d})$, we have
\begin{align*}
\langle \mathfrak{T}_m(T)\star S,\Psi\rangle=&\,\langle \F_\sigma(\mathfrak{T}_m(T)\star S),\F_\sigma(\Psi)\rangle\\
=&\,\langle \F_W(\mathfrak{T}_m(T))\F_W(S),\F_\sigma(\Psi)\rangle\\
=&\,\langle m\F_W(T)\F_W(S),\F_\sigma(\Psi)\rangle\\
=&\,\langle m\F_\sigma(T\star S),\F_\sigma(\Psi)\rangle\\
=&\,\langle \F_\sigma(T_m(T\star S)),\F_\sigma(\Psi)\rangle
= \langle T_m(T\star S),\Psi \rangle.
\end{align*}
The assertion then holds as $\Psi$ is an arbitrary Schwartz function.
\end{proof}
\begin{remark}
Lemma \ref{lem:Commutation} formally follows from the fact that Fourier multipliers are convolution operators. If $K$ denotes the distributional symplectic Fourier transform of $m$, then
\[
\mathfrak{T}_m(T)\star S=(K\star T)\star S=K*(T\star S)=T_m(T\star S).
\]
\end{remark}

The integral in \eqref{multiDef} is well-defined as a bounded linear operator on $L^2(\R^d)$ by Moyal's identity. Namely, given any $f,g\in L^2(\R^d)$, it follows that
\[
|\langle\mathfrak{T}_m(T)f,g\rangle|=\left|\int_{\R^{2d}}m(z)\F_W(T)(z)\langle \rho(z)f,g\rangle\,dz\right|\leq \|m\|_{L^\infty}\|\F_W(T)\|_{L^2}\|\mathcal{A}(g,f)\|_{L^2},
\]
by Cauchy-Schwarz' inequality. The isometric isomorphism of the Fourier-Wigner transform from the space of Hilbert-Schmidt operators $\cS^2$ to $L^2(\R^{2d})$, together with the inclusion properties of Schatten class operators, then yields
\[
\|\mathfrak{T}_m(T)\|_{L^2(\R^d)\to L^2(\R^d)}\leq \|m\|_{L^\infty(\R^{2d})}\|T\|_{\cS^2}\leq \|m\|_{L^\infty(\R^{2d})}\|T\|_{\cS^p},
\]
for any $1\leq p\leq 2$. As the trace class operators are dense in $\cS^p$ for this range of $p$, it follows by density that the Fourier-Wigner multiplier $\mathfrak{T}_m$ maps $T\in\cS^p$ to a bounded operator on $L^2$ for any 
choice of $m\in L^\infty(\R^{2d})$.

We denote the set of all $m\in L^\infty(\R^{2d})$ for which $\mathfrak{T}_m$ extends to a bounded operator from $\cS^p$ to $\cS^q$ by $\mathfrak{M}_{p,q}$, and whenever $p=q$ we simplify to $\mathfrak{M}_{p}$. By the previous paragraph, it follows that $L^\infty(\R^{2d})\cong \mathfrak{M}_{2,\infty}$, as any Hilbert-Schmidt operator maps to a bounded operator on $L^2(\R^d)$.

The definition of Fourier-Wigner multipliers is the correct definition in terms of the Weyl quantisation. The next lemma shows that the Fourier-Wigner multiplier commutes with the Weyl quantisation in terms of classical symplectic Fourier multipliers.
\begin{lemma}
Let $T\in\cS^p$, and $m\in \mathfrak{M}_{p,q}$. Then the Weyl symbol of $\mathfrak{T}_{m}(T)$ is given by
\[
a_{\mathfrak{T}_mT}=T_ma_{T}.
\]
\end{lemma}
\begin{proof}
For any $f,g\in\sS(\R^d)$, it follows that
\begin{align*}
\langle \mathfrak{T}_m T f,g\rangle_{\sS',\sS}
=\langle m\F_W(T),\mathcal{A}(g,f)\rangle_{\sS',\sS}
=&\,\langle m\F_\sigma(a_T),\mathcal{A}(g,f)\rangle_{\sS',\sS}\\
=&\,\langle T_m(a_T), \mathcal{W}(g,f)\rangle_{\sS',\sS}\\
=&\,\langle L_{T_m(a_T)}f,g \rangle_{\sS',\sS},
\end{align*}
where we have used that $\F_W(T)=\F_\sigma(a_T)$.
\end{proof}

If we restrict to rank-one operators $T=\varphi\otimes \psi\in \cS^1$ for $\varphi,\psi\in\sS(\R^d)$, then $T\in \fS$ and
\[
\mathfrak{T}_m(T)=\F_\sigma(m)\star(\varphi\otimes \psi)=\mathcal{A}_{\F_\sigma(m)}^{\psi,\varphi},
\]
is the classical mixed-state localisation operator with symbol $\F_\sigma(m)$ and windows $\varphi,\psi$. If we now assume $m\in\mathfrak{M}_{p,q}$, then it follows that
\[
\|\F_\sigma(m)\star(\varphi\otimes \psi)\|_{\cS^q}\leq \|m\|_{\mathfrak{M}_{p,q}}\|\varphi\|_{L^2}\|\psi\|_{L^2}\leq\|m\|_{\mathfrak{M}_{p,q}}\|\varphi\|_{M^{1,1}}\|\psi\|_{M^{1,1}}
\]
for every $\varphi,\psi\in \sS(\R^d)$. As such, by Theorem \ref{thm:NessSchattenCond} it follows that $\F_\sigma(m)\in M^{q,\infty}(\R^{2d})$.
\begin{prop}\label{prop:ModulSpace}
If $m\in\mathfrak{M}_{p,q}$, then $\F_\sigma(m)\in M^{q,\infty}(\R^{2d})$.
\end{prop}

In fact, we can say even more for the Hilbert-Schmidt case. Due to the fact that the Weyl quantisation is an isometric $*$-isomorphism from $L^{2}(\R^{2d})$ to $\cS^2$, it easily follows for any $m\in\mathfrak{M}_2$ and all $T\in \cS^2$ that
\[
\|\mathfrak{T}_m(T)\|_{\cS^2}=\|m\F_W(T)\|_{L^{2}(\R^{2d})}=\|m\F_\sigma(a_T)\|_{L^{2}(\R^{2d})}=\|T_m(a_T)\|_{L^{2}(\R^{2d})},
\]
where $a_T$ denotes the Weyl symbol of $T$. This necessarily implies $\mathfrak{M}_2=\mathcal{M}_2$, and
\[
\|m\|_{\mathfrak{M}_2}=\|m\|_{\mathcal{M}_2}=\|m\|_{L^{\infty}(\R^{2d})}.
\]
We summarise this in the following proposition.
\begin{prop}\label{2-multipliers}
The space $\mathfrak{M}_2$ is isomorphic to $L^\infty(\R^{2d})$ with
\[
\|m\|_{\mathfrak{M}_2}=\|m\|_{\mathcal{M}_2}=\|m\|_{L^{\infty}(\R^{2d})}.
\]
\end{prop}
Even though the multipliers coincide whenever $p=q=2$, we do not expect this to happen in general. One reason for this is the inclusion properties of Schatten classes which is missing for $L^p$ spaces. Namely, for each $1\leq p,q\leq \infty$ and each $m\in\mathfrak{M}_{p,q}$, it follows that
\[
\|\mathfrak{T}_mS\|_{\cS^q}\leq \|m\|_{\mathfrak{M}_{p,q}}\|S\|_{\cS^p}\leq \|m\|_{\mathfrak{M}_{p,q}}\|S\|_{\cS^r},\qquad \|\mathfrak{T}_mS\|_{\cS^s}\leq \|\mathfrak{T}_mS\|_{\cS^q}\leq \|m\|_{\mathfrak{M}_{p,q}}\|S\|_{\cS^p},
\]
for every $r\leq p$ and $s\geq q$. Writing this as a proposition gives.
\begin{prop}\label{prop:nesting}
Let $1\leq p,q\leq \infty$. Then for any $r\leq p$, and $s\geq q$ it follows that $\mathfrak{M}_{p,q}\subseteq \mathfrak{M}_{r,q}$
and $\mathfrak{M}_{p,q}\subseteq \mathfrak{M}_{p,s}$, and
\[
\|m\|_{\mathfrak{M}_{r,q}}\leq \|m\|_{\mathfrak{M}_{p,q}},\qquad \|m\|_{\mathfrak{M}_{p,s}}\leq \|m\|_{\mathfrak{M}_{p,q}}.
\]
\end{prop}
This implies that $\mathcal{M}_{p,q}$ and $\mathfrak{M}_{p,q}$ cannot in general coincide whenever $q\geq p$ by simply considering the multiplier corresponding to the identity operator.

\begin{corollary}
Let $1\leq p< \infty$ and $q>p$. Then there exists $m\in L^\infty(\R^{2d})$ such that $m\in\mathfrak{M}_{p,q}$, but $m\notin \mathcal{M}_{p,q}$. In particular $\mathfrak{M}_{p,q}\neq \mathcal{M}_{p,q}$ for $q>p$.
\end{corollary}
\begin{proof}
Consider $m\equiv 1\in L^\infty(\R^{2d})$. Then $\F_\sigma(1)=\delta_0$, and for any $T\in \cS^p$,
\[
\langle \mathfrak{T}_1(T), S\rangle_{\fS',\fS}=\langle \F_W(T),\F_W(S)\rangle_{\sS',\sS}=\langle T,S\rangle_{\fS',\fS}
\]
for every $S\in\fS$ by Definition \ref{defDistFW}. This shows that $T=\mathfrak{T}_1(T)$. In particular, for any $q>p$,
\[
\|\mathfrak{T}_1(T)\|_{\cS^q}=\|T\|_{\cS^q}\leq \|T\|_{\cS^p},
\]
and so $1\in \mathfrak{M}_{p,q}$. 

As $L^p(\R^{2d})$ and $L^q(\R^{2d})$ are not equal, nor comparable, it follows that there exists $F\in L^p(\R^{2d})$ such that $F\notin L^q(\R^{2d})$. Let $\{\Phi_n\}_{n=1}^\infty\subseteq \sS(\R^{2d})$ be a sequence converging to $F$ in $L^p(\R^{2d})$. By renormalising and choosing an adequate subsequence, we may assume that
\[
\|\Phi_n\|_{L^q}\geq n,\qquad \|\Phi_n\|_{L^p}\leq1.
\]
The multiplier $m\equiv 1$ corresponds to convolving with $\delta_0$, and thus is the identity operator on $\sS(\R^{2d})$. Thus for each $n\in\mathbb{N}$ it follows that,
\[
\|T_1(\Phi_n)\|_{L^q}=\|\Phi_n\|_{L^q}\geq n\geq n\|\Phi_n\|_{L^p}.
\]
This shows that $1\notin \mathcal{M}_{p,q}$.
\end{proof}
\begin{remark}
For a more concrete counter-example, we may consider an appropriate sequence of Gaussians. For each $n\in\mathbb{N}$ consider the function
\[
\Phi_n(z)=e^{-\pi n |z|^2}\in\sS(\R^{2d}).
\]
Then the $L^p$ norm is given by
\[
\|\Phi_{n}\|_{L^p}^p=\int_{\R^{2d}}e^{-\pi np |z|^2}\,dz=\left(\frac{1}{np}\right)^d.
\]
In particular, as $\Phi_n\in\sS(\R^{2d})$ it follows that $T_1(\Phi_n)=\Phi_n*\delta_0=\Phi_n$. We thus conclude that for each $n\in\mathbb{N}$,
\[
\|1\|_{\mathfrak{M}_{p,q}}\geq \frac{\|\Phi_n\|_{L^q}}{\|\Phi_n\|_{L^p}}=\frac{p^\frac{d}{p}}{q^{\frac{d}{q}}}n^{\frac{d}{p}-\frac{d}{q}},
\]
which shows that $1\notin\mathcal{M}_{p,q}$ as $1/p>1/q$.
\end{remark}

If we instead consider the case when $q<p$, then the Fourier-Wigner multipliers turn out to coincide with the symplectic Fourier multipliers. As is known from the classical literature (see for instance \cite{Grafakos_14}), these multipliers are all equal to the trivial multiplier.
\begin{prop}
$\mathfrak{M}_{p,q}=\{0\}$ whenever $1\leq q<p<\infty$.
\end{prop}
\begin{proof}
Assume $m\in\mathfrak{M}_{p,q}$ for $1\leq q<p<\infty$, and let $\varphi_0(t)=2^{d/4}\exp(-\pi|t|^2)$ denote the $L^2$ normalized Gaussian on $\R^d$. Then for any $F\in L^p(\R^{2d})$, it follows that
\begin{equation}\label{eq:classicalMultiEx}
F\mapsto \mathfrak{T}_m(F\star(\varphi_0\otimes \varphi_0))\star (\varphi_0\otimes \varphi_0)\in L^q(\R^{2d}),
\end{equation}
by Werner-Young's inequality. However, utilising Proposition \ref{ConvProp-Distribution} and \eqref{eq:FW_Ambi}, coupled with \eqref{Ambi-Gauss}, results in
\[
\F_\sigma\left(\mathfrak{T}_m(F\star(\varphi_0\otimes \varphi_0))\star (\varphi_0\otimes \varphi_0)\right)(z)=\F_\sigma(F)(z)\left(m(z)e^{-\pi \frac{|z|^2}{2}}e^{-\pi\frac{|z|^2}{2}}\right)=k(z)\F_\sigma(f)(z).
\]
This shows that the map given by \eqref{eq:classicalMultiEx}
is a symplectic Fourier multiplier from $L^p(\R^{2d})$ to $L^q(\R^{2d})$. Hence it follows that for every $z\in\R^{2d}$,
\[
m(z)e^{-\pi|z|^2}=k(z)= 0,
\]
by Theorem $2.5.6$ in \cite{Grafakos_14}. Since the Gaussian never vanishes, it follows that $m\equiv 0$.
\end{proof}

For $m\in L^\infty(\R^{2d})$ let $\widetilde{m}(z)=m(-z)$ denote the parity of $m$. We reserve the notation $P$ to denote the parity operator on $L^2(\R^d)$. Given an operator $T\in\cS^p$, we consider the adjoint of the operator $\mathfrak{T}_m(T)$. This leaves us with the following result.
\begin{lemma}\label{lem:adjoint}
Let $m\in\mathfrak{M}_{p,q}$. Then for any $T\in \cS^p$ the adjoint of $\mathfrak{T}_mT$ is given by $(\mathfrak{T}_mT)^*=\mathfrak{T}_{\widetilde{\overline{m}}}T^*$.
\end{lemma}
\begin{proof}
It suffices to assume that $T\in\cS^1$. For any $f,g\in \sS(\R^d)$ it then follows that
\begin{align*}
\langle \mathfrak{T}_m T f,g\rangle=\int_{\R^{2d}}m(z)\F_W(T)(z)\langle \rho(z) f,g\rangle \,dz
=&\, \overline{\int_{\R^{2d}}\overline{m(z)}\overline{\F_W(T)(z)}\langle \rho(-z)g,f\rangle\,dz}\\
=&\, \overline{\int_{\R^{2d}}\overline{m(z)}\F_W(T^*)(-z)\langle \rho(-z)g,f\rangle\,dz}\\
=&\, \langle f,\mathfrak{T}_{\widetilde{\overline{m}}}T^* g\rangle,
\end{align*}
where we used that $\F_W(T)(z)=\overline{\F_W(T^*)(-z)}$.
\end{proof}

\begin{remark}
This is reminiscent of the fact that
\[
\overline{T_mf}(x)=T_{\widetilde{\overline{m}}}\overline{f}(x),
\]
for classical Fourier multipliers.
\end{remark}

\begin{prop}\label{prop:Duality}
The spaces $\mathfrak{M}_{p,q}$ and $\mathfrak{M}_{q',p'}$ are isometric with the property that
\[
\|m\|_{\mathfrak{M}_{p,q}}=\|m\|_{\mathfrak{M}_{q',p'}}.
\]
\end{prop}
\begin{proof}
Let $S,T\in\cS^1\subseteq \cS^2$. Then
\[
\langle \mathfrak{T}_mS, T\rangle_{\cS^q,\cS^{q'}}=\int_{\R^{2d}}m\F_W(S)\overline{\F_W(T)}\,dz=\int_{\R^{2d}}\F_W(S)\overline{\overline{m}\F_W(T)}\,dz=\langle S,\mathfrak{T}_{\overline{m}}T\rangle_{\cS^{p},\cS^{p'}}.
\]
Since $\cS^1$ is dense in $\mathcal{K}$ and $\cS^p$ for $1<p< \infty$ we can, by the continuity of the dual pairing, extend the estimate to hold for any $S\in\cS^p$ and $T\in \cS^{q'}$, or $T\in\mathcal{K}$ for $q=1$. This shows that $m\in\mathfrak{M}_{p,q}$ if and only if $\overline{m}\in\mathfrak{M}_{q',p'}$ as
\[
\|m\|_{\mathfrak{M}_{p,q}}=\|\mathfrak{T}_m\|_{\cS^{p}\to\cS^q}=\|\mathfrak{T}_{\overline{m}}\|_{\cS^{q'}\to\cS^{p'}}=\|\overline{m}\|_{\mathfrak{M}_{q',p'}}.
\]

On the other hand, by Lemma \ref{lem:adjoint} it follows that if $m\in\mathfrak{M}_{p,q}$, then 
\[
\mathfrak{T}_{\widetilde{\overline{m}}}S^*=(\mathfrak{T}_{m} S)^*\in \cS^q,
\]
for every $S\in \cS^p$ as taking the adjoint leaves the singular values unchanged. 

We claim that $\widetilde{m}\in \mathfrak{M}_{p,q}$ if and only if $m\in\mathfrak{M}_{p,q}$. To show this, we note that $\rho(-z)=P\rho(z)P$, where $P$ is the parity operator on $L^2(\R^d)$.
It follows that for any operator $T\in\cS^1$,
\[
\mathfrak{T}_{\widetilde{m}}(PTP)=\int_{\R^{2d}}m(-z)\F_W(T)(-z)\rho(z)\,dz=\int_{\R^{2d}}m(z)\F_W(T)(z)P\rho(z)P\,dz=P\mathfrak{T}_{m}(T)P.
\]

Since the Schatten $p$-classes are operator ideals, it follows that $PTP\in\cS^p$ if and only if $T\in \cS^p$. Thus, for any operator $T\in\cS^1$, we have
\begin{align*}
\|\mathfrak{T}_{\widetilde{m}}(T)\|_{\cS^q}\leq \|P\|^2_{L^2\to L^2}\|\mathfrak{T}_m(PTP)\|_{\cS^q}
\leq&\, \|P\|^{2}_{L^2\to L^2}\|m\|_{\mathfrak{M}_{p,q}}\|PTP\|_{\cS^p}\\
\leq&\, \|P\|^{4}_{L^2\to L^2}\|m\|_{\mathfrak{M}_{p,q}}\|T\|_{\cS^p}.
\end{align*}
As $\cS^1$ is dense in $\cS^p$, we can extend $\mathfrak{T}_{\widetilde{m}}$ to a continuous operator from $\cS^p$ to $\cS^q$. This necessarily implies that
\[
\|\widetilde{m}\|_{\mathfrak{M}_{p,q}}\leq \|m\|_{\mathfrak{M}_{p,q}}<\infty,
\]
as $\|P\|_{L^2\to L^2}=1$, and thus $\widetilde{m}\in \mathfrak{M}_{p,q}$. In fact, by replacing $m$ by $\widetilde{m}$ we see that
\begin{equation}\label{eq:RemoveTilde}
\|\widetilde{m}\|_{\mathfrak{M}_{p,q}}= \|m\|_{\mathfrak{M}_{p,q}}.
\end{equation}
For $m\in\mathfrak{M}_{p,q}$ and any $T\in\cS^{p'}$ it then follows that
\[
\|\mathfrak{T}_{m}(T)\|_{\cS^{q'}}=\|(\mathfrak{T}_{m}(T))^*\|_{\cS^{q'}}
=\|\mathfrak{T}_{\widetilde{\overline{m}}}(T^*)\|_{\cS^{q'}}\leq \|\widetilde{\overline{m}}\|_{\mathfrak{M}_{q',p'}}\|T^*\|_{\cS^{p'}}=\|\overline{m}\|_{\mathfrak{M}_{q',p'}}\|T\|_{\cS^{p'}},
\]
where we used \eqref{eq:RemoveTilde} and the fact that singular values are invariant under taking adjoints. This implies that
\[
\|m\|_{\mathfrak{M}_{q',p'}}\leq\|\overline{m}\|_{\mathfrak{M}_{q',p'}}=\|m\|_{\mathfrak{M}_{p,q}}.
\]
For the other direction, we have for any $T\in \cS^{p'}$,
\[
\|\mathfrak{T}_{\overline{m}}(T)\|_{\cS^{q'}}=\|(\mathfrak{T}_{\overline{m}}(T))^*\|_{\cS^{q'}}=\|\mathfrak{T}_{\widetilde{m}}(T^*)\|_{\cS^{q'}}\leq \|\widetilde{m}\|_{\mathfrak{M}_{q',p'}}\|T^*\|_{\cS^{p'}}=\|m\|_{\mathfrak{M}_{q',p'}}\|T\|_{\cS^{p'}}.
\]
Thus, by the definition of the Fourier-Wigner multiplier norm we have
\[
\|m\|_{\mathfrak{M}_{p,q}}=\|\overline{m}\|_{\mathfrak{M}_{q',p'}}\leq
\|m\|_{\mathfrak{M}_{q',p'}},
\]
which gives equality and finishes the proof.
\end{proof}

\begin{corollary}\label{cor:embedding}
Let $1\leq p<q\leq 2$. We have the inclusions
\[
\mathfrak{M}_1\subseteq \mathfrak{M}_p \subseteq \mathfrak{M}_q\subseteq \mathfrak{M}_2\cong L^\infty(\R^{2d}),
\]
with the estimate
\[
\|m\|_{\mathfrak{M}_{2}}\leq \|m\|_{\mathfrak{M}_{q}} \leq \|m\|_{\mathfrak{M}_{p}}\leq \|m\|_{\mathfrak{M}_{1}}.
\]
\end{corollary}
\begin{proof}
Let $m\in\mathfrak{M}_p$, and thus $m\in\mathfrak{M}_{p'}$ by Proposition \ref{prop:Duality}. Thus, by Riesz-Thorin interpolation, Theorem \ref{thm:RTInter}, it follows that
\[
\|m\|_{\mathfrak{M}_q}=\|\mathfrak{T}_m\|_{\cS^q\to\cS^q}\leq \|\mathfrak{T}_m\|_{\cS^p\to\cS^p}^{1-\theta}\|\mathfrak{T}_m\|_{\cS^{p'}\to\cS^{p'}}^{\theta}=\|\mathfrak{T}_m\|_{\cS^p\to\cS^p}=\|m\|_{\mathfrak{M}_p},
\]
for any $1\leq p<q\leq 2$.
\end{proof}

\begin{prop}
$\mathfrak{M}_{p}$ is a Banach algebra for each $1\leq p \leq \infty$.
\end{prop}
\begin{proof}
The vector space structure follows directly from the linearity of the Fourier-Wigner transform. Moreover, for each $m_1,m_2\in\mathfrak{M}_p$, it follows for any $T\in\cS^p\subseteq \fS'$ that the operator $\mathfrak{T}_{m_1m_2}$ satisfies
\begin{align*}
\langle \mathfrak{T}_{m_1m_2}(T),S\rangle_{\fS',\fS}=\langle m_1m_2\F_W(T),\F_W(S)\rangle_{\sS',\sS}
=&\,\langle m_1\F_W(\mathfrak{T}_{m_2}(T)),\F_W(S)\rangle_{\sS',\sS}\\
=&\,\langle \mathfrak{T}_{m_1}(\mathfrak{T}_{m_2}(T)),S\rangle_{\fS',\fS},
\end{align*}
for every $S\in\fS$. This shows that $\mathfrak{T}_{m_1m_2}$ is the composition of $\mathfrak{T}_{m_1}$ and $\mathfrak{T}_{m_2}$, and thus $m_1m_2\in\mathfrak{M}_{p}$ whenever $m_1,m_2\in\mathfrak{M}_p$.

By Proposition \ref{prop:Duality} we may assume that $1\leq p\leq 2$.
To show that $\mathfrak{M}_p$ is complete, we note that we can identify $\mathfrak{M}_{p}$ as a subspace of $\mathcal{L}(\cS^p)$. Thus, it suffices to show that $\mathfrak{M}_{p}$ is closed.
Let $\{m_j\}_{j\in\mathbb{N}}$ be a Cauchy sequence in $\mathfrak{M}_{p}$. By the completeness of $\mathcal{L}(\cS^p)$, there exists $\mathfrak{T}\in \mathcal{L}(\cS^p)$ such that $\mathfrak{T}_{m_j}$ converges to $\mathfrak{T}$ with respect to the operator norm.
By Corollary \ref{cor:embedding}, the sequence is also a Cauchy sequence in $L^\infty(\R^{2d})$, and thus there exists $m\in L^\infty(\R^{2d})$ such that $m_j\to m$ in $L^\infty(\R^{2d})$.

For any $T\in \cS^p\subseteq \cS^2$ and any $f,g\in L^2(\R^d)$, it follows by Cauchy-Schwarz' inequality that for each $j\in\mathbb{N}$,
\[
\left|\int_{\R^{2d}}m_j(z)\F_W(T)(z)\mathcal{A}(f,g)(z)\,dz\right|
\leq \sup_{n\in\mathbb{N}}\|m_n\|_{L^\infty}\|\F_W(T)\|_{L^2(\R^{2d})}\|\mathcal{A}(f,g)\|_{L^2(\R^{2d})}<\infty.
\]
The dominated convergence theorem then gives,
\begin{align*}
\langle \mathfrak{T}(T) f,g\rangle
 =&\,\lim_{n\to\infty}\int_{\R^{2d}}m_n(z)\F_W(T)(z)\mathcal{A}(f,g)(z)\,dz\\
=&\,\int_{\R^{2d}}m(z)\F_W(T)(z)\mathcal{A}(f,g)(z)\,dz
=\langle \mathfrak{T}_m(T) f,g\rangle,
\end{align*}
for any $f,g\in L^2(\R^d)$, which shows that $\mathfrak{T}=\mathfrak{T}_m$. The expression on the right hand side is well-defined as $T\in\cS^p\subseteq \cS^2$ and $\mathfrak{T}_m:\cS^2\to\cS^2$ continuously whenever $m\in L^\infty(\R^{2d})$ by Theorem \ref{2-multipliers}.

It only remains to show that $m\in\mathfrak{M}_{p}$. Let $S\in \cS^1$ be any trace class operator. Then for any fixed $T\in \cS^p\subseteq \cS^2$,
\[
\left|\langle \mathfrak{T}_m(T)-\mathfrak{T}_{m_j}(T),S\rangle\right|\leq \|m-m_j\|_{L^\infty}\|T\|_{\cS^2}\|S\|_{\cS^2}\leq \|m-m_j\|_{L^\infty}\|T\|_{\cS^p}\|S\|_{\cS^1}\xrightarrow{j\to\infty}0.
\]
However, this implies that $\|\mathfrak{T}_m(T)-\mathfrak{T}_{m_j}(T)\|_{L^2\to L^2}$ converges to zero, and thus $\mathfrak{T}_m(T)$ is a compact operator where the singular values  satisfy $s_n(\mathfrak{T}_{m_j}(T))\to s_n(\mathfrak{T}_{m}(T))$ for each $n$ as $j\to\infty$. By Fatou's lemma we have
\begin{align*}
\|\mathfrak{T}_m(T)\|_{\cS^p}^p
=\sum_{n=1}^\infty |s_n(\mathfrak{T}_{m}(T))|^p
\leq \liminf_{j\to\infty}\sum_{n=1}^\infty |s_n(\mathfrak{T}_{m_j}(T))|^p
=&\,\liminf_{j\to\infty}\|\mathfrak{T}_{m_j}(T)\|_{\cS^p}^p\\
\leq&\, \liminf_{j\to\infty}\|m_j\|_{\mathfrak{M}_{p}}^p\|T\|_{\cS^p}^p.
\end{align*}
Since $\{m_j\}_{j\in\mathbb{N}}$ is a Cauchy sequence in $\mathfrak{M}_{p}$, it follows that
\[
\liminf_{j\to\infty}\|m_j\|_{\mathfrak{M}_{p}}\leq \sup_{n\in\mathbb{N}}\|m_n\|_{\mathfrak{M}_{p}}<\infty,
\]
which shows that
\[
\|m\|_{\mathfrak{M}_{p}}=\sup_{\|T\|_{\cS^p}=1}\|\mathfrak{T}_m(T)\|_{\cS^p}\leq \liminf_{j\to\infty}\|m_j\|_{\mathfrak{M}_{p}}<\infty,
\]
and so $m\in\mathfrak{M}_{p}$.
\end{proof}

We mention the following boundedness result for multipliers in Lorentz spaces, originally proved by Ruzhansky and Tulenov in the setting of noncommutative Euclidean spaces \cite{Ruzhansky2024multipliers}.  The proof is added for completeness.
\begin{theorem}[\cite{Ruzhansky2024multipliers}, Theorem. $3.3$]
Let $1<p\leq 2\leq q<\infty$ and suppose $m\in L^{r,\infty}(\R^{2d})$ for $r^{-1}=p^{-1}-q^{-1}$. Then $m\in \mathfrak{M}_{p,q}$ with $\|m\|_{\mathfrak{M}_{p,q}}\leq C\|m\|_{L^{r,\infty}}$ for some $C=C(p,q)>0$.
\end{theorem}
\begin{proof}
The proof follows by a straightforward calculation. By the reverse quantum Hausdorff-Young's inequality, followed by H\"{o}lder's inequality for Lorentz spaces, it follows that for any $T\in \cS^p$
\begin{align*}
\|\mathfrak{T}_mT\|_{\cS^q}
\leq C\|m\F_W(T)\|_{L^{q',q}}
\leq C\|m\|_{L^{r,\infty}}\|\F_W(T)\|_{L^{p',q}}
\leq&\,C\, \|m\|_{L^{r,\infty}}\|\F_W(T)\|_{L^{p',p}}\\
\leq&\, C\|m\|_{L^{r,\infty}}\|T\|_{\cS^{p}}.
\end{align*}
\end{proof}

\section{Proof of Theorem \ref{thm:MainThm}}
We are now ready to present the proof of Theorem \ref{thm:MainThm}. The proof builds on Lemma \ref{LW-bnd-cor}, which allows us to bound the $\cS^p$ norm of an operator with compact Fourier-Wigner support in terms of the $L^p$ norm of the convolution with a rank-one operator. Using that Fourier-Wigner multipliers commutes with convolutions then gives that any classical $L^p$ multiplier with compact Fourier support is also a Fourier-Wigner multiplier on $\cS^p$. The reverse implication is proved in a similar fashion.
\begin{proof}[Proof of Theorem \ref{thm:MainThm}]
Let $m\in L^\infty(\R^{2d})$ be compactly supported with $\text{supp}(m)\subseteq \Omega$ for some compact set $\Omega$.

Let $1\leq p\leq \infty$ and assume $m\in \mathcal{M}_{p,q}$. For $T\in \cS^p$ note that
\[
\F_W(\mathfrak{T}_m(T))=m\F_W(T),
\]
is supported on $\Omega$. Thus, by Lemma \ref{LW-bnd-cor} it follows that there exist $L^2$-normalised $f,g\in \sS(\R^d)$ and $C=C(\Omega)>0$ such that
\[
\|\mathfrak{T}_m(T)\|_{\cS^q}\leq C(\Omega)\,\|\mathfrak{T}_m(T)\star (f\otimes g)\|_{L^q(\R^{2d})}.
\]
By Lemma \ref{lem:Commutation} and the fact that $m\in \mathcal{M}_p$, we also have
\[
\|\mathfrak{T}_m(T)\star (f\otimes g)\|_{L^q(\R^{2d})}=\|T_m(T\star (f\otimes g))\|_{L^q(\R^{2d})}\leq \|m\|_{\mathcal{M}_{p,q}}\|T\star(f\otimes g)\|_{L^p(\R^{2d})}.
\]
Combining these estimates with Werner-Young's inequality results in
\[
\|\mathfrak{T}_m(T)\|_{\cS^q}\leq C(\Omega)\|m\|_{\mathcal{M}_{p,q}}\|T\|_{S^p},
\]
which shows that $m\in\mathfrak{M}_{p,q}$.

Assume that $m\in \mathfrak{M}_{p,q}$. Then for any $F\in L^p(\R^{2d})$, it follows by Lemma \ref{lem:Commutation} and Lemma \ref{LW-bnd-cor}  that
\begin{align*}
\|T_m(F)\|_{L^q(\R^{2d})}\leq C(\Omega)\|\mathfrak{T}_m(F\star (f\otimes g))\|_{\cS^q}
\leq&\, C(\Omega)\|m\|_{\mathfrak{M}_{p,q}}\|F\star(f\otimes g)\|_{\cS^p}\\
\leq&\, C(\Omega)\|m\|_{\mathfrak{M}_{p,q}}\|F\|_{L^p(\R^{2d})},
\end{align*}
where the last inequality follows from Werner-Young's inequality and the fact that $f$ and $g$ are $L^2$ normalised. This shows that $m\in \mathcal{M}_{p,q}$. 
\end{proof}

\begin{remark}
Similarly to the results in \cite{Luef_Samuelsen_24,Samuelsen_24}, the proof of Theorem \ref{thm:MainThm} shows that for a compactly supported function $m\in L^{\infty}(\R^{2d})$, there exists a constant $C\geq 1$ such that
\[
\frac{1}{C}\|m\|_{\mathfrak{M}_{p,q}}\leq \|m\|_{\mathcal{M}_{p,q}}\leq C\|m\|_{\mathfrak{M}_{p,q}},
\]
where the constant $C$ depends only on the support of $m$. 
\end{remark}

\begin{remark}
Note that Theorem \ref{thm:MainThm} necessarily implies that any compactly supported Fourier multiplier from $L^p(\R^{2d})$ to $L^{q}(\R^{2d})$ is also a Fourier multiplier from $L^p(\R^{2d})$ to $L^s(\R^{2d})$ for any $s\geq q$. This is a consequence of Proposition \ref{prop:nesting}. However, this can also be shown classically using the Bernstein inequality \cite[Chapter $5$]{Wolff}. Namely, if $m$ is supported on a ball of radius $R$, then $m\F_\sigma(f)$ is also supported on the same ball. Thus, for any $s\geq q$, Bernstein's inequality gives
\[
\|T_mf\|_{L^{s}(\R^{2d})}\leq CR^{2d\left(\frac{1}{q}-\frac{1}{s}\right)}\|T_mf\|_{L^q(\R^{2d})}\leq CR^{2d\left(\frac{1}{q}-\frac{1}{s}\right)}\|m\|_{\mathcal{M}_{p,q}}\|f\|_{L^p(\R^{2d})}.
\]

Perhaps more surprising is the fact that for $r\leq p$, it follows that any compactly supported multiplier from $L^p(\R^{2d})$ to $L^q(\R^{2d})$ must also be a Fourier multiplier from $L^r(\R^{2d})$ to $L^q(\R^{2d})$. For the operator case, the Schatten classes are nested, but this is not the case for general $L^p$ spaces. However, the mapping properties are known by \cite[Corollary $1.8$]{Hormander_60}.
Namely, if $m\in\mathcal{M}_{p,q}$, then for $\varphi\in\sS(\R^{2d})$ it follows that $m\varphi\in\mathcal{M}_{r,q}$ for every $r\leq p$. Since any compactly supported function $m$ can be written as $m=m\chi$, where $\chi\in C_c^{\infty}$ is a smooth cut-off function equal to $1$ on the support of $m$, the result follows.
\end{remark}

\section{Trace class multipliers}
We end with a discussion on trace class multipliers. Unlike the Hilbert-Schmidt multipliers, we are unable to present a characterisation of the trace class multipliers.
A natural characterisation, if there is one, would be the Fourier transform of complex Radon measures as in classical $L^1$ theory. It is true that the symplectic Fourier transform of any complex Radon measure necessarily gives a trace class multiplier. This follows by an extension of Young's inequality for measures. It is shown in \cite{Feichtinger_Halvdansson_Luef_24} that
\[
\|T\star\mu\|_{\cS^1}\leq \|T\|_{\cS^1}\|\mu\|_{\mathscr{M}},
\]
for any complex Radon measure, and thus $\mathcal{M}_1\subseteq \mathfrak{M}_1$. Moreover, by Proposition \ref{prop:ModulSpace} it follows that $\F_\sigma(\mathfrak{M}_1)\subseteq M^{1,\infty}(\R^{2d})$. To check if a function is a trace class multipliers, it is enough to only consider rank-one operators.
\begin{lemma}\label{lem:TraceOnRank1}
Let $m\in L^\infty(\R^d)$. Then $m\in \mathfrak{M}_1$ if and only if there exists $C>0$ such that
\[
\|\mathcal{A}_{\F_\sigma(m)}^{f,g}\|_{\cS^1}\leq C\|f\|_{L^2}\|g\|_{L^2},
\]
for all $f,g\in L^2(\R^d)$.
\end{lemma}
\begin{proof}
Since $\mathcal{A}_{\F_\sigma(m)}^{f,g}=\mathfrak{T}_{m}(f\otimes g)$ it follows for any $m\in \mathfrak{M}_1$ that
\[
\|\mathcal{A}_{\F_\sigma(m)}^{f,g}\|_{\cS^1}\leq  \|m\|_{\mathfrak{M}_1}\|f\otimes g\|_{\cS^1}=\|m\|_{\mathfrak{M}_1}\|f\|_{L^2}\|g\|_{L^2}.
\]

For the other direction, we note that for any finite rank operator $T$ it follows from the singular value decomposition and the triangle inequality that
\[
\|\mathfrak{T}_m(T)\|_{\cS^1}\leq \sum_{n=1}^N |s_n(T)|\|\mathcal{A}_{\F_\sigma(m)}^{e_n,\eta_n}\|_{\cS^1}\leq C\sum_{n=1}^N |s_n(T)|= C\|T\|_{\cS^1}.
\]
This shows that $\mathfrak{T}_m$ is a bounded linear operator on a dense subspace of $\cS^1$, and thus has an extension to bounded linear map on $\cS^1$. This implies that $m\in \mathfrak{M}_1$
\end{proof}
Though we cannot show that the trace class multipliers coincide with the classical $L^1$ multipliers there are some similarities, such that any $m\in \mathfrak{M}_1$ has to be a bounded continuous function on phase space.
\begin{prop}\label{lem:ContBnd}
If $m\in \mathfrak{M}_1$, then $m\in C_b(\R^{2d})$.
\end{prop}
\begin{proof}
Let $m\in\mathfrak{M}_1$ and $\Gamma_1(z)=\exp{(-\pi|z|^2)}$. Consider the Weyl quantisation of $\Gamma_1$ given by
\[
L_{\Gamma_1}=\int_{\R^{2d}}\F_\sigma(\Gamma_1)(z)\rho(z)\,dz=\int_{\R^{2d}}e^{-\pi |z|^2}\rho(z)\,dz\in\fS.
\]
As $m\in\mathfrak{M}_1$, it follows by Werner-Young's inequality that
\[
\mathfrak{T}_m(L_{\Gamma_1})\star L_{\Gamma_1}\in L^1(\R^{2d}).
\]
Taking the symplectic Fourier transform, it follows that
\[
F(z)=\F_{\sigma}(\mathfrak{T}_m(L_{\Gamma_1})\star L_{\Gamma_1})=m(z)e^{-2\pi|z|^2}\in C_0(\R^{2d}),
\]
by the Riemann-Lebesgue lemma. In particular, as the Gaussian does not vanish on $\R^{2d}$, we conclude that
\[
m(z)=e^{2\pi|z|^2}F(z)\in C(\R^{2d}).
\]
Since $m\in L^\infty(\R^{2d})$ by assumption, it follows that $m\in C_b(\R^{2d})$.
\end{proof}

The proof of Proposition \ref{lem:ContBnd} holds true whenever $L_{\Gamma_1}$ is replaced by
\[
\Gamma_\varepsilon(z)=\varepsilon^{-2d}e^{-\pi\frac{|z|^2}{\varepsilon^2}},
\]
for any $\varepsilon>0$. Moreover,
 the $L^1$ function defined by $\mathfrak{T}_m(L_{\Gamma_\varepsilon})\star L_{\Gamma_\varepsilon}$ can be written as
 \[
 \F_\sigma(m)*(L_{\Gamma_\varepsilon}\star L_{\Gamma_\varepsilon}),
 \]
 by the associativity of the convolutions. Here $\F_\sigma(m)\in \sS'(\R^{2d})$ is the distributional symplectic Fourier transform of the multiplier $m\in C_b(\R^{2d})$. By \cite[Proposition $9.10$]{Folland-Real}, it follows that
 \[
 \mathfrak{T}_m(L_{\Gamma_\varepsilon})\star L_{\Gamma_\varepsilon}= \F_\sigma(m)*(L_{\Gamma_\varepsilon}\star L_{\Gamma_\varepsilon})\in C^\infty(\R^{2d})\cap L^1(\R^{2d})\subseteq C^\infty_0(\R^{2d}).
 \]
 In particular
 \[
 \int_{\R^{2d}} \mathfrak{T}_m(L_{\Gamma_\varepsilon})\star L_{\Gamma_\varepsilon}(z)\,dz=\F_\sigma( \mathfrak{T}_m(L_{\Gamma_\varepsilon})\star L_{\Gamma_\varepsilon})(0)=m(0),
 \]
 which is completely independent of $\varepsilon$.
 Since $L_{\Gamma_\varepsilon}\star L_{\Gamma_\varepsilon}$ is an approximate identity as $\varepsilon\to 0$, it is also true that
 \[
 \langle \F_\sigma(m)*(L_{\Gamma_\varepsilon}\star L_{\Gamma_\varepsilon}),\Psi\rangle\to \langle \F_\sigma(m),\Psi\rangle,
 \]
 for any $\Psi\in \sS(\R^{2d})$ as $\varepsilon\to 0$. Since $L^1$ can be embedded into $\mathscr{M}(\R^{2d})=C_0'(\R^{2d})$, it would follow by compactness that $\F_\sigma(m)$ is a complex Radon measure if we could show that the operators $ \mathfrak{T}_m(L_{\Gamma_\varepsilon})\star L_{\Gamma_\varepsilon}$ are uniformly bounded in the $L^1$-norm as $\varepsilon\to 0$. Unfortunately, we cannot apply Werner-Young's inequality as the operators $L_{\Gamma_\varepsilon}$ are not uniformly bounded in trace norm. This is the content of the next lemma.
 
 \begin{lemma}\label{lem:ApproxId}
For any $\varepsilon>0$, let $\Gamma_{\varepsilon}$ denote the dilated Gaussian given by
\[
\Gamma_\varepsilon(z)=\varepsilon^{-2d}e^{-\pi\frac{|z|^2}{\varepsilon^2}}.
\]
The operator $L_{\Gamma_\varepsilon}$ is a trace class operator with the property that $\{L_{\Gamma_\varepsilon}\star L_{\Gamma_\varepsilon}\}_{\varepsilon>0}$ is an approximate identity. Moreover, the operator $L_{\Gamma_\varepsilon}$ is a positive operator if and only if $\varepsilon^2\geq 2^{-1}$, in which case
\[
\|L_{\Gamma_\varepsilon}\|_{\cS^1}=1.
\]
If $\varepsilon^2< 2^{-1}$, then
\[
\|L_{\Gamma_\varepsilon}\|_{\cS^1}\geq (2\varepsilon^2)^{-\frac{d}{2}}.
\]
\end{lemma}
\begin{proof}
The trace class property follows from the fact that $\Gamma_\varepsilon\in\sS(\R^{2d})$, while the approximate identity follows from the fact that
\[
L_{\Gamma_\varepsilon}\star L_{\Gamma_\varepsilon}(z)=\Gamma_\varepsilon*\Gamma_\varepsilon(z),
\]
and that $\Gamma_\varepsilon$ is an approximate identity as $\varepsilon\to 0$. For positivity, it follows by \cite[Theorem $21$]{Cordero_DeGosso_Nicola_19} that the Weyl quantisation of the Gaussian
\[
\Gamma_{\Sigma}(z)=(2\pi)^{-d}\sqrt{\det{\Sigma^{-1}}}e^{-\frac{1}{2}\langle \Sigma^{-1}z,z\rangle}
\] is a positive trace class operator if and only if
\[
\frac{1}{4\pi}\leq \lambda_\text{min},
\]
where $\lambda_{\text{min}}$ is the smallest eigenvalue of $\Sigma$. Here $\Sigma$ is a positive symmetric $2d\times 2d$ matrix. Note that $\Gamma_\varepsilon=\Gamma_\Sigma$ with
\[
\Sigma=\mathrm{diag}\left(\frac{\varepsilon^2}{2\pi}\right),
\]
so $L_{\Gamma_\varepsilon}$ is positive if and only if
\[
\frac{1}{4\pi}\leq \frac{\varepsilon^2}{2\pi}.
\]
If $L_{\Gamma_\varepsilon}$ is positive, then
\[
\|L_{\Gamma_\varepsilon}\|_{\cS^1}=\tr(L_{\Gamma_\varepsilon})=\F_W(L_{\Gamma_\varepsilon})(0)=\F_\sigma(\Gamma_{\varepsilon})(0)=\int_{\R^{2d}}\Gamma_\varepsilon(z)\,dz=1.
\]
On the other hand, by the inclusion of Schatten classes it follows that
\[
\|L_{\Gamma_\varepsilon}\|_{\cS^1}\geq \|L_{\Gamma_\varepsilon}\|_{\cS^2}=\|\Gamma_\varepsilon\|_{L^2}=(2\varepsilon^2)^{-\frac{d}{2}}.
\]
\end{proof}
\begin{remark}
Note that for any $\Phi,\Psi\in\sS(\R^{2d})$, it follows that
\[
\langle L_{\Gamma_\varepsilon}\Phi,\Psi\rangle=\langle \Gamma_{\varepsilon},\mathcal{W}(\Psi,\Phi)\rangle\xrightarrow{\varepsilon\to 0}\langle \delta,\mathcal{W}(\Psi,\Phi)\rangle=\overline{\mathcal{W}(\Psi,\Phi)(0)}=2^{d}\langle P\Phi,\Psi\rangle,
\]
which shows that $L_{\Gamma_\varepsilon}$ converges to $2^dP\notin \cS^1$ in the weak operator topology as $\varepsilon\to 0$.
\end{remark}

 Let us note a second property of $\F_\sigma(m)$ for $m\in \mathfrak{M}_1$ which is also true for measures, namely for any $\Psi,\Phi\in\sS(\R^{2d})$, the $C^\infty$ function defined by
 \[
 \F_\sigma(m)*(\Psi*\Phi)=\F_\sigma(m)*(L_\Psi\star L_\Phi)=\mathfrak{T}_m(L_\Psi)\star L_{\Phi},
 \]
is in $L^1(\R^{2d})$ by Werner-Young's inequality. However, If $\tau\in\sS'(\R^{2d})$ and $\Phi\mapsto \tau*\Phi$ maps $\mathscr{S}(\R^{2d})$ to $L^1(\R^{2d})$, then it is not necessarily true that $\tau\in\mathscr{M}(\R^{2d})$. Consider for instance
$\tau=\partial^\alpha\delta$ for some multi-index $\alpha\in \mathbb{N}_0^{2d}$. Then
\[
\tau*\Phi=\partial^\alpha\Phi\in L^1(\R^{2d}),
\]
as the Schwartz class is closed under differentiation. The symplectic Fourier transform of $\tau$ is given by
\[
\F_\sigma(\tau)(\zeta)=(-2\pi i J\zeta)^\alpha,
\]
which is not a bounded function on $\R^{2d}$, and thus is excluded by Proposition \ref{lem:ContBnd}. Here $J$ is the standard symplectic matrix on $\R^{2d}$. 

Another example of a tempered distribution which maps the Schwartz class to $L^1$ through convolution is the $C_b(\R^{2d})$ function
\begin{equation}\label{eq:TauSymbol}
\tau(z)=\frac{1}{2i}\left(e^{\pi i \frac{d}{2}}e^{-\pi i |z|^2}-e^{-\pi i \frac{d}{2}}e^{\pi i |z|^2}\right).
\end{equation}
The function $\tau$ corresponds to the distributional symplectic Fourier transform of the function $z\mapsto \sin(\pi |z|^2)$. We refer to \cite[Theorem $7.6.1$]{Hormander-ALPDO1} or \cite[Proposition $4.2$]{Wolff} for the distributional Fourier transform of complex Gaussians. The derivatives of $\sin(\pi |z|^2)$ are of at most polynomial growth, and thus, for any Schwartz function $\Phi\in\sS(\R^{2d})$, it follows that
\begin{align*}
\left|z^\alpha\frac{\partial^\beta}{\partial z^\beta}\Phi(z)\sin(\pi |z|^2)\right|\leq&\, \sum_{j\leq |\beta|}|z^\alpha P_j(z)\Phi^{(j)}(z)|,
\end{align*}
for any $\alpha,\beta\in \mathbb{N}^{2d}_0$, where $P_j$ is a polynomial for each $j$. This shows that $\Phi(z)\sin(\pi |z|^2)$ is a Schwartz function whenever $\Phi\in\sS(\R^{2d})$. From here we conclude that
\[
\Phi*\tau=\F_\sigma(\F_\sigma(\Phi)\sin(\pi|\cdot|^2))\in \sS(\R^{2d})\subseteq L^1(\R^{2d}),
\]
but $\tau\notin \mathscr{M}(\R^{2d})$. However, it is unclear whether or not $\sin(\pi |z|^2)$ is a trace class multiplier. This question is equivalent to asking if mixed-state localisation operator with symbol $\tau$ given by \eqref{eq:TauSymbol} is a trace class operator for any choice of window functions $f,g\in L^2(\R^d)$ by Lemma \ref{lem:TraceOnRank1}. We therefore end with an open question.
\begin{question}
For each $m\in\mathfrak{M}_1$, does there exist $\mu\in \mathscr{M}(\R^{2d})$ such that $m=\F_\sigma(\mu)$?
\end{question}

\section*{Acknowledgements}
This work was partially conducted during a research stay funded through the NSF grant DMS-2247185 at Stanford University in the Autumn of 2024. The author is grateful for the hospitality of Stanford University and Professor Eugenia Malinnikova. The author would also like to thank Josef Greilhuber for insightful discussions, and Professor Franz Luef and Associate Professor Sigrid Grepstad for their helpful comments on earlier versions of this manuscript.

\printbibliography

\end{document}